\documentclass[12pt,oneside,english]{amsart}
\usepackage[T1]{fontenc}
\usepackage[latin9]{inputenc}
\usepackage{geometry}
\usepackage{amstext}
\usepackage{amsthm}
\usepackage{amssymb}
\usepackage{setspace}
\makeatletter
\numberwithin{equation}{section}
\numberwithin{figure}{section}
\theoremstyle{plain}
\newtheorem{thm}{\protect\theoremname}
\theoremstyle{plain}
\newtheorem{conjecture}[thm]{\protect\conjecturename}
\theoremstyle{plain}
\newtheorem{cor}[thm]{\protect\corollaryname}
\theoremstyle{plain}
\newtheorem{lem}[thm]{\protect\lemmaname}
\theoremstyle{plain}
\newtheorem{remark}[thm]{\protect\remarkname}

\usepackage{babel}
\providecommand{\conjecturename}{Conjecture}
\providecommand{\corollaryname}{Corollary}
\providecommand{\lemmaname}{Lemma}
\providecommand{\theoremname}{Theorem}
\providecommand{\remarkname}{Remark}

\providecommand{\theoremname}{Remark}

\makeatother

\usepackage{babel}
\providecommand{\conjecturename}{Conjecture}
\providecommand{\corollaryname}{Corollary}
\providecommand{\lemmaname}{Lemma}
\providecommand{\theoremname}{Theorem}

\begin{document}
\title{On Diameters of Cayley Graphs over Matrix Groups}
\author{Eitan Porat}
\begin{abstract}
We establish that for the matrix groups $G=\mathrm{GL}_{n}\left(\mathbb{F}_{p}\right)$
or $G=\mathrm{SL}_{n}\left(\mathbb{F}_{p}\right)$, there exist absolute
constants $c\in\left(0,1\right)$ and $C>0$ such that any symmetric
generating set $A$ with $\left|A\right|\geq\left|G\right|^{1-c}$
has a covering number $\leq Cn^{2}$. This result is sharp up to the
value of the constant $C>0$. 
\end{abstract}

\maketitle
\tableofcontents{}

\section{Introduction}

For a finite group $G$ and a symmetric generating set $A\subseteq G$
(in the sense that $A^{-1}=A$), the Cayley graph $\mathrm{Cay}\left(G,A\right)$
is defined as a connected graph with vertex set $G$ and edge set
\[
\left\{ \left(g,ag\right)\mid g\in G,a\in A\right\} .
\]
The diameter of this graph is the smallest integer $k$ such that
every element of $G$ can be expressed as a product of at most $k$
elements from $A$. The diameter of the group is denoted by $\mathrm{diam}\left(G\right)$,
and it represents the maximum among all Cayley graph diameters $\mathrm{Cay}\left(G,A\right)$
where $A$ runs through all generating sets of $G$.

This becomes an intrinsic property of $G$, rather than being tied
to the Cayley graph associated with $G$ or any specific generating
set.

The diameter of $\mathrm{Cay}\left(G,A\right)$ is referred to as
the covering number of $A$, denoted by $\mathrm{cn}\left(A\right)=\min\left\{ k:A^{k}=G\right\} $.
Babai \cite{babai1992diameter} proposed the following conjecture: 
\begin{conjecture}
If $G$ is a non-Abelian finite simple group, then $\mathrm{diam}\left(G\right)=\left(\log\left|G\right|\right)^{O\left(1\right)}$. 
\end{conjecture}

This conjecture remains one of the major unresolved problems in combinatorics.
Liebeck and Shalev \cite{liebeck2001diameters} confirmed the conjecture
when the generating set $A$ is a normal set. The conjecture was proven
by Helfgott~\cite{helfgott2008growth} for the case $G=\mathrm{SL}_{2}\left(\mathbb{F}_{p}\right)$,
and subsequently, it was verified for finite simple groups of Lie
type and bounded rank independently by Pyber and Szabo~\cite{pyber2016growth}
and Breuillard, Green, and Tao~\cite{breuillard2012structure}. It
remains to prove Babai's conjecture for alternating groups and for
classical groups of unbounded rank. A special case of this conjecture
was proven by Halasi \cite{halasi2021diameter} for the case where
$G=\mathrm{SL}_{n}\left(\mathbb{F}_{p}\right)$ and $A$ is a generating
set that includes a \emph{transvection}, defined as a matrix of the
form $I+x$, where $x$ is a rank 1 matrix.

In this paper, we prove a special case of this conjecture for the
scenario where $G$ is the general linear group $\mathrm{GL}_{n}\left(\mathbb{F}_{p}\right)$
and $A$ is a sufficiently large generating set. 
\begin{thm}
If $\left|A\right|\geq\left|\mathrm{GL}_{n}\left(\mathbb{F}_{p}\right)\right|^{1-c}$
for some $0<c<1$, then for some absolute constant $C>0$, $\mathrm{cn}(A)\leq Cn^{2}$. 
\end{thm}

\begin{remark} The proof can be readily adapted to work with $G=\mathrm{SL}_{n}\left(\mathbb{F}_{p}\right)$.
\end{remark}

\begin{remark} Although $\mathrm{SL}_{n}\left(\mathbb{F}_{p}\right)$
is not a simple group, the theorem also holds for $\mathrm{PSL}_{n}\left(\mathbb{F}_{p}\right)$.
\end{remark}

This result is sharp up to the absolute constant $C$. Our proof relies
on the concept of a \emph{groumvirate}, which is a conjugate of the
natural embedding of $\mathrm{\mathrm{G}L}_{n-t}\left(\mathbb{F}_{p}\right)$
in $\mathrm{\mathrm{G}L}_{n}\left(\mathbb{F}_{p}\right)$, i.e., the
subgroup 
\[
\left\{ \begin{pmatrix}I_{t\times t} & 0\\
0 & X
\end{pmatrix}\mid X\in\mathrm{\mathrm{G}L}_{n-t}\left(\mathbb{F}_{p}\right)\right\} 
\]
where $0\leq t\leq n$. We define $\mu\left(A\right):=\frac{\left|A\right|}{\left|\mathrm{\mathrm{G}L}_{n}\left(\mathbb{F}_{p}\right)\right|}$
as the uniform measure on $\mathrm{GL}_{n}\left(\mathbb{F}_{p}\right)$. 
\begin{thm}[Evra, Kindler, and Lifshitz \cite{evra2024polynomial}]
There exists a $d>0$ such that for every subset $A$ of $\mathrm{SL}_{n}\left(\mathbb{F}_{p}\right)$,
the set $AA^{-1}AA^{-1}$ contains a groumvirate with density at least
$\mu\left(A\right)^{d}$. 
\end{thm}

For a given $t$, 
\[
\mu\left(\left\{ \begin{pmatrix}I_{t\times t} & 0\\
0 & X
\end{pmatrix}\mid X\in\mathrm{GL}_{n-t}\left(\mathbb{F}_{p}\right)\right\} \right)\leq\frac{q^{\left(n-t\right)^{2}}}{\left|\mathrm{\mathrm{G}L}_{n}\left(\mathbb{F}_{p}\right)\right|}=\frac{q^{\left(n-t\right)^{2}}}{q^{n^{2}}\left(q-1\right)^{n}}\leq q^{\left(n-t\right)^{2}-n^{2}}
\]
Thus, to contain a \emph{groumvirate} on $n-t$ variables, the density
of $A$ must exceed, 
\[
\mu\left(A\right)\geq q^{\frac{\left(n-t\right)^{2}-n^{2}}{d}}
\]
That is 
\[
\left|A\right|\geq q^{\frac{\left(n-t\right)^{2}-n^{2}}{d}}q^{n^{2}}=q^{\left(1-\frac{1}{d}\right)n^{2}+\frac{1}{d}\left(n-t\right)^{2}}
\]
For $t=\varepsilon n$, 
\[
\left|A\right|\geq q^{\left(1-\left(\frac{1-\left(1-\varepsilon\right)^{2}}{d}\right)\right)n^{2}}=q^{\left(1-c_{\varepsilon}\right)n^{2}}=\Omega\left(\left|\mathrm{GL}_{n}\left(\mathbb{F}_{p}\right)\right|^{1-c_{\varepsilon}}\right)
\]
where $c_{\varepsilon}=\frac{1-\left(1-\varepsilon\right)^{2}}{d}$.
The bound is non-trivial if $\varepsilon>0$. For the remainder of
the paper, we take $\varepsilon=\frac{1}{4}$ and $\left|A\right|\geq\left|\mathrm{GL}_{n}\left(\mathbb{F}_{p}\right)\right|^{1-c_{\varepsilon}}$,
and we will often omit explicitly stating this assumption.

In this paper, we assume that the \emph{groumvirate} is expressed
in the standard basis, i.e., $\left\{ \begin{pmatrix}I_{t\times t} & 0\\
0 & X
\end{pmatrix}\mid X\in\mathrm{GL}_{n-t}\left(\mathbb{F}_{p}\right)\right\} =I_{t}\otimes\mathrm{GL}_{n-t}\left(\mathbb{F}_{p}\right)$ for some $t$. Since $A$ is symmetric, $AA^{-1}AA^{-1}=A^{4}$.
We may assume without loss of generality that $I\in A$ (this does
not alter the asymptotics since $I\in AA^{-1}$). 
\begin{cor}
\label{groumvirate}Suppose $\left|A\right|\geq\left|\mathrm{GL}_{n}\left(\mathbb{F}_{p}\right)\right|^{1-c_{\varepsilon}}$,
$A^{4}\supseteq\left\{ \begin{pmatrix}I_{\epsilon n} & 0\\
0 & X
\end{pmatrix}\mid X\in\mathrm{\mathrm{GL}}_{(1-\varepsilon)n}\left(\mathbb{F}_{p}\right)\right\} $, up to a change of basis. 
\end{cor}

This groumvirate acts as $\mathrm{GL}_{n-t}\left(\mathbb{F}_{p}\right)$
on a subspace of $\mathbb{F}_{p}^{n}$. Our goal is to upgrade it
to an action on the whole space.

We define the standard basis to be $e_{1}=\begin{pmatrix}1\\
0\\
\vdots\\
0
\end{pmatrix},e_{2}=\begin{pmatrix}0\\
1\\
\vdots\\
0
\end{pmatrix},\dots,e_{n}=\begin{pmatrix}0\\
0\\
\vdots\\
1
\end{pmatrix}$. The span of a set of vectors $v_{1},\dots,v_{k}$ is denoted by
$\left\langle v_{1},\dots,v_{k}\right\rangle $. For a set $S\subseteq[n]$,
denote $e_{S}=\lbrace e_{i}:i\in S\rbrace$ and $e_{\bar{S}}=\lbrace e_{i}:i\notin S\rbrace$.
We define the projection of a vector $v$ onto $\left\langle e_{1},\dots,e_{t}\right\rangle $
by $\pi_{1}\left(v\right)$ and its projection onto $\left\langle e_{t+1},\dots,e_{n}\right\rangle $
by $\pi_{2}\left(v\right)$. Additionally, we can modify the $\pi_{2}$
component of vectors using $M\in A^{4}$. 
\begin{lem}
\label{corr} Let $\lbrace v_{1},\dots,v_{k}\rbrace\subseteq\left\langle e_{t+1},\dots,e_{n}\right\rangle $
and $\lbrace v'_{1},\dots,v'_{k}\rbrace\subseteq\left\langle e_{t+1},\dots,e_{n}\right\rangle $,
for $k\leq n-t$. Assume that $\lbrace v_{1},\dots,v_{k}\rbrace$
is a linearly independent set and $\lbrace v_{1}',\dots,v_{k}'\rbrace$
is also linearly independent; then there exists $M\in A^{4}$ such
that $v_{i}=v'_{i}$, which fixes $\lbrace e_{1},\dots,e_{t}\rbrace$. 
\end{lem}

\begin{proof}
Since $\left\{ v_{1},\dots,v_{k}\right\} \subseteq\left\langle e_{t+1},\dots,e_{n}\right\rangle $
and $\lbrace v_{1}',\dots,v_{k}'\rbrace\subseteq\left\langle e_{t+1},\dots,e_{n}\right\rangle $
are linearly independent sets, there exists an invertible linear transformation
$T:\left\langle e_{t+1},\dots,e_{n}\right\rangle \to\left\langle e_{t+1},\dots,e_{n}\right\rangle $
such that $T\left(v_{i}\right)=v_{i}'$ for every $1\leq i\leq k$.
We can extend $T$ to a transformation $T':\mathbb{F}_{p}^{n}\to\mathbb{F}_{p}^{n}$
such that $T\left(e_{i}\right)=e_{i}$ for every $1\leq i\leq t$.
Define $M$ to be a matrix representation of $T'$ in the standard
basis, that is $M=\left[T'\right]_{\left\langle e_{1},\dots,e_{n}\right\rangle }^{\left\langle e_{1},\dots,e_{n}\right\rangle }$.
By construction, $M\in I_{t}\otimes\mathrm{GL}_{n-t}(\mathbb{F_{p}})\subseteq A^{4}$. 
\end{proof}
Our main contribution is the following Theorem: 
\begin{thm}
\label{useful Theorem} Let $\mathcal{V}=\left\lbrace (U,W)\mid U\oplus W=\mathbb{F}_{p}^{n}\right\rbrace $,
and define an action $\cdot:A\times\mathcal{V}\to\mathcal{V}$ by
$g\cdot(U,W)=(gU,gW)$. Consider the Schreier graph corresponding
to this action on the vertex set $\mathcal{V}$. Then, any two vertices
in $\mathcal{V}'=\left\lbrace (\langle e_{S}\rangle,\langle e_{\bar{S}}\rangle)\mid S\subseteq[n],|S|=t\right\rbrace $
are at most $O(t^{2})$ steps apart in this graph. 
\end{thm}

Intuitively, this allows us to upgrade the action of the groumvirate
to an arbitrary set of size $n-t$ of columns or rows of our choosing. 
\begin{cor}
\label{finalcor} Any transformation $X\in\mathrm{GL}_{n}(\mathbb{F}_{p})$
which fixes $e_{S}$ and is invariant on $\langle e_{\bar{S}}\rangle$
for $|S|=t$, can be constructed in $O(t^{2})$ steps. 
\end{cor}

\section{Upgrading The Groumvirate's Action}
\begin{lem}
\label{lem8}Let $V\subseteq\left\langle e_{t+1},\dots,e_{n}\right\rangle $
be a vector space of dimension $\dim\left(V\right)=k$. Let $\left\{ u_{i}\right\} _{i=1}^{m}\cup\left\{ w_{i}\right\} _{i=1}^{\ell}\subseteq\left\langle e_{t+1},\dots,e_{n}\right\rangle $
be a linearly independent set, where $m+\ell\leq k$. There exists
$T\in A^{4}$ such that $Tu_{i}=u_{i}$ for every $1\leq i\leq m$,
$Tw_{i}\in V$ for every $1\leq i\leq\ell$, and $Te_{i}=e_{i}$ for
every $1\leq i\leq t$. 
\end{lem}

\begin{proof}
Let $\left\{ v_{1},\dots,v_{k}\right\} $ be a basis of $V$. Since
$\dim\left(\left\langle u_{1},\dots,u_{m}\right\rangle \right)=m\leq k$,
there exists a vector $v_{1}\in V$ such that $\dim\left\langle v,u_{1},\dots\right\rangle =m+1$.
We inductively add vectors $v_{1},\dots,v_{k-m}$ until we obtain
a linearly independent set $\left\{ u_{1},\dots,u_{m},v_{1},\dots,v_{k-m}\right\} $.
Applying Lemma \ref{corr}, there exists a transformation such that
$Tu_{i}=u_{i}$ for every $1\leq i\leq m$, $Tw_{i}=v_{i}$ for every
$1\leq i\leq\ell$, and $Te_{i}=e_{i}$ for every $1\leq i\leq t$. 
\end{proof}
\begin{lem}
\label{lemma9}Let $V$ be a vector space and $a\in\mathrm{GL}_{n}\left(\mathbb{F}_{p}\right)$.
Consider the subspace $W\subseteq V$ of $v\in V$ such that $av\in V$,
then $\dim\left(W\right)\geq2\dim\left(V\right)-n$. 
\end{lem}

\begin{proof}
By Grassmann's Lemma, 
\begin{align*}
\dim\left(W\right) & =\dim\left(V\cap a^{-1}V\right)=\dim\left(V\right)+\dim\left(a^{-1}V\right)-\dim\left(V+a^{-1}V\right)\\
 & =2\dim\left(V\right)-\dim\left(V+a^{-1}V\right)\geq2\dim\left(V\right)-n.
\end{align*}
\end{proof}
\begin{lem}
\label{main Lemma}There exists $a\in A^{O(t^{2})}$ such that $ae_{i}=e_{t+i}$
for $1\leq i\leq t$. 
\end{lem}

\begin{proof}[Proof]
\begin{enumerate}
\item We use Lemma \ref{Lemma9} to construct vectors $v_{1},\dots,v_{t}\in\left\langle e_{t+1},\dots,e_{n}\right\rangle $
and $a_{i}\in A^{i}$ and $b_{t}\in A^{O(t^{2})}$ such that $\left\{ \pi_{1}\left(a_{i}v_{i}\right)\right\} _{i=1}^{t}$
is a basis for $\left\langle e_{1},\dots,e_{t}\right\rangle $, $\left\{ \pi_{2}\left(a_{i}v_{i}\right)\right\} _{i=1}^{t}$
is a linearly independent set, and $b_{t}a_{i}v_{i}\in\left\langle e_{t+1},\dots,e_{n}\right\rangle $
for every $1\leq i\leq t$. 
\item We use Lemma \ref{Lemma8} to construct a transformation $c_{t}\in A^{O(t^{2})}$
such that $\left\{ \pi_{2}\left(c_{t}e_{i}\right)\right\} _{i=1}^{t}$
is a linearly independent set. There exist $\left\{ z_{1},\dots z_{t}\right\} \subseteq\left\langle e_{t+1},\dots,e_{n}\right\rangle $
such that $\left\{ \pi_{2}\left(a_{i}v_{i}\right)\right\} _{i=1}^{t}\cup\left\{ z_{1},\dots z_{t}\right\} $
is a basis of $\left\langle e_{t+1},\dots,e_{n}\right\rangle $. By
Lemma \ref{corr}, there exists $T_{1}\in A^{4}$ such that $T\pi_{2}\left(c_{t}e_{i}\right)=z_{i}$
for $1\leq i\leq t$. Hence, we may generally assume that $\left\{ \pi_{2}\left(c_{t}e_{i}\right)\right\} _{i=1}^{t}\cup\left\{ \pi_{2}\left(a_{i}v_{i}\right)\right\} _{i=1}^{t}$
is a linearly independent set. 
\item We express $c_{t}e_{i}$ as $c_{t}e_{i}=\sum_{j=1}^{t}\alpha_{ij}a_{j}v_{j}+w_{i}$
with $\alpha_{ij}\in\mathbb{F}_{p}$ and $w_{i}\in\left\langle e_{t+1},\dots,e_{n}\right\rangle $
(since $\lbrace\pi_{1}(a_{i}v_{i})\rbrace_{i=1}^{t}$ is a basis of
$\left\langle e_{1},\dots,e_{t}\right\rangle $). By construction
$\left\{ \pi_{2}\left(c_{t}e_{i}\right)\right\} _{i=1}^{t}\cup\left\{ \pi_{2}\left(a_{i}v_{i}\right)\right\} _{i=1}^{t}$
is linearly independent. Since elementary operations do not change
the linear independence of a set, $\left\{ \pi_{2}\left(w_{i}\right)\right\} _{i=1}^{t}\cup\left\{ \pi_{2}\left(a_{i}v_{i}\right)\right\} _{i=1}^{t}$
is also linearly independent. By Lemmas \ref{lem8} and \ref{lemma9},
there exists a transformation $T_{2}\in A^{4}$ such that $T_{2}w_{i}=w_{i}'\in b_{t}^{-1}\left\langle e_{t+1},\dots,e_{n}\right\rangle \cap\left\langle e_{t+1},\dots,e_{n}\right\rangle $,
$T_{2}\pi_{2}\left(a_{i}v_{i}\right)=\pi_{2}\left(a_{i}v_{i}\right)$,
and $T_{2}e_{i}=e_{i}$ for $1\leq i\leq t$. 
\item Since $b_{t}a_{j}v_{j}\in\left\langle e_{t+1},\dots,e_{n}\right\rangle $
and $b_{t}w_{i}'\in\left\langle e_{t+1},\dots,e_{n}\right\rangle $,
\[
b_{t}\left(\sum_{j=1}^{t}\alpha_{ij}a_{j}v_{j}+w_{i}'\right)\in\left\langle e_{t+1},\dots,e_{n}\right\rangle .
\]
\item We apply lemma \ref{corr} to construct a transformation $T_{3}\in I_{t}\otimes\mathrm{GL}_{n}(\mathbb{F}_{p})$
such that $T_{3}b_{t}T_{2}T_{1}c_{t}e_{i}=e_{t+i}$ for every $1\leq i\leq t$. 
\end{enumerate}
\end{proof}
\begin{lem}
\label{Lemma9} There exist $\lbrace v_{i}\rbrace_{i=1}^{t}\subseteq\left\langle e_{t+1},\dots,e_{n}\right\rangle $,
$\lbrace a_{i}\rbrace_{i=1}^{t}\subseteq A^{t}$ and $b\in A^{O(t^{2})}$
such that the following hold: 
\begin{enumerate}
\item $\left\{ \pi_{1}\left(a_{i}v_{i}\right)\right\} _{i=1}^{t}$ is a
basis for $\left\langle e_{1},\dots,e_{t}\right\rangle $. 
\item $\left\{ \pi_{2}\left(a_{i}v_{i}\right)\right\} _{i=1}^{t}$ is a
linearly independent set. 
\item $\lbrace ba_{i}v_{i}\rbrace_{i=1}^{t}\subseteq\left\langle e_{t+1},\dots,e_{n}\right\rangle $. 
\end{enumerate}
\end{lem}

\begin{proof}
Part 1 is proven by induction on $t$. Since $A$ is a generating
set, $\left\langle e_{t+1},\dots,e_{n}\right\rangle $ is not invariant
under $A$, thus there exists $v_{1}\in\left\langle e_{t+1},\dots,e_{n}\right\rangle $
and $a_{1}\in A$ such that $\pi_{1}\left(a_{1}v_{1}\right)\neq0$.
For the induction step, since $\left\langle a_{1}v_{1},\dots,a_{i}v_{i},e_{t+1},\dots,e_{n}\right\rangle $
is not invariant under $A$, there exists $v\in\left\langle a_{1}v_{1},\dots,a_{i}v_{i},e_{t+1},\dots,e_{n}\right\rangle $
and $a_{i+1}\in A$ such that $a_{i+1}v\notin\left\langle a_{1}v_{1},\dots,a_{i}v_{i},e_{t+1},\dots,e_{n}\right\rangle $.
Therefore, there exists $v_{i+1}=e_{i}$ for some $t+1\leq i\leq n$
or $v_{i+1}=a_{j}v_{j}\in A^{j}\left\langle e_{t+1},\dots,e_{n}\right\rangle $
such that $a_{i+1}v_{i+1}\notin\left\langle a_{1}v_{1},\dots,a_{i}v_{i},e_{t+1},\dots,e_{n}\right\rangle $,
therefore $a_{i+1}v_{i+1}\in A^{i+1}\left\langle e_{t+1},\dots,e_{n}\right\rangle $.
Hence, $\left\{ \pi_{1}\left(a_{i}v_{i}\right)\right\} _{i=1}^{t}$
is a linearly independent set.

Part 2 is proven by induction. By Lemma \ref{Lemma9},

\[
\dim\left(a_{i}^{-1}\left\langle e_{t+1},\dots,e_{n}\right\rangle \cap\left\langle e_{t+1},\dots,e_{n}\right\rangle \right)\geq t
\]
for every $1\leq i\leq t$.

Thus, there exists $u_{1}\in\left\langle e_{t+1},\dots,e_{n}\right\rangle $
such that $\pi_{2}(a_{1}u_{1})\neq0$, set $v_{1}'=v_{1}+u_{1}$.
Assume we have constructed $v_{1}',\dots,v_{i}'$ such that $\lbrace\pi_{2}(a_{j}v_{j}')\rbrace_{j=1}^{i}$
is a linearly independent set. For the induction step, if $\pi_{2}(a_{i+1}v_{i+1})$
is linearly independent of $\lbrace\pi_{2}(a_{j}v_{j}')\rbrace_{j\leq i}$,
set $v_{i+1}'=v_{i+1}$. Otherwise, set $v_{i+1}'=v_{i+1}+u_{i+1}$
where $u_{i+1}\in a_{i+1}^{-1}\left\langle e_{t+1},\dots,e_{n}\right\rangle \cap\left\langle e_{t+1},\dots,e_{n}\right\rangle $
and $\pi_{2}(a_{i+1}u_{i+1})$ is linearly independent of $\lbrace\pi_{2}(a_{j}v_{j}')\rbrace_{j\leq i}$.
We replace $\lbrace v_{1},\dots,v_{t}\rbrace$ with $\lbrace v_{1}',\dots,v_{t}'\rbrace$
and proceed to the last part.

For Part 3, the construction of $b=b_{t}$ is done inductively, where
the induction is on the number of vectors $\lbrace a_{1}v_{1},\dots,a_{i}v_{i}\rbrace$
such that $\lbrace b_{i}a_{1}v_{1},\dots,b_{i}a_{i}v_{i}\rbrace\subseteq\left\langle e_{t+1},\dots,e_{n}\right\rangle $.
Initially, we are given a set $\left\{ a_{1}v_{1},\dots,a_{t}v_{t}\right\} $
where $\lbrace a_{1}v_{1},\dots,a_{t}v_{t}\rbrace\cap\left\langle e_{t+1},\dots,e_{n}\right\rangle =\emptyset$
and $\left\{ \pi_{2}\left(a_{i}v_{i}\right)\right\} _{i=1}^{t}$ is
linearly independent. For the base case, we use Lemma \ref{corr}
to add $w_{i}\in a_{1}\left\langle e_{t+1},\dots,e_{n}\right\rangle \cap\left\langle e_{t+1},\dots,e_{n}\right\rangle $
to $a_{i}v_{i}$ for every $1\le i\leq t$, where $\left\{ w_{i}\right\} _{i=1}^{t}$
are chosen such that $\left\{ \pi_{2}\left(a_{1}^{-1}\left(a_{i}v_{i}+w_{i}\right)\right)\right\} _{i=1}^{t}$
is a linearly independent set. After adding $\left\{ w_{i}\right\} _{i=1}^{t}$,
we apply $a_{1}^{-1}$ to all the vectors. We are left with the set
\[
\left\{ v_{1}+a_{1}^{-1}w_{1},a_{1}^{-1}a_{2}v_{2}+a_{1}^{-1}w_{2},\dots,a_{1}^{-1}a_{t}v_{t}+a_{1}^{-1}w_{t}\right\} 
\]
We define $d_{1}=v_{1}+a_{1}^{-1}w_{1}\in\left\langle e_{t+1},\dots,e_{n}\right\rangle $.
Note that $\left\{ \pi_{2}\left(a_{1}^{-1}\left(a_{i}v_{i}+w_{i}\right)\right)\right\} _{i=1}^{t}$
is a linearly independent set.

Now assume we are given 
\[
\left\{ d_{1},\dots,d_{i},a_{i}^{-1}a_{i+1}v_{i+1}+w_{i+1},\dots,a_{i}^{-1}a_{t}v_{t}+w_{t}\right\} 
\]
where $d_{1},\dots,d_{i},w_{i+1},\dots,w_{t}\in\left\langle e_{t+1},\dots,e_{n}\right\rangle $
and 
\[
\left\{ \pi_{2}\left(d_{1}\right),\dots,\pi_{2}\left(d_{i}\right)\right\} \cup\left\{ \pi_{2}\left(a_{i}^{-1}a_{i+1}v_{i+1}+w_{i+1}\right),\dots,\pi_{2}\left(a_{i}^{-1}a_{t}v_{t}+w_{t}\right)\right\} 
\]
is a linearly independent set. By Lemma \ref{lem8} and \ref{lemma9},
there exists a transformation $M\in I_{t}\otimes\mathrm{GL}_{n-t}\left(\mathbb{F}_{p}\right)$
such that $Md_{j}\in a_{i}^{-1}a_{i+1}\left\langle e_{t+1},\dots,e_{n}\right\rangle \cap\left\langle e_{t+1},\dots,e_{n}\right\rangle $
for every $1\leq i\leq j$ such that $M\pi_{2}\left(a_{i}^{-1}a_{i+1}v_{j}+w_{i+1}\right)=\pi_{2}\left(a_{i}^{-1}a_{i+1}v_{j}+w_{i+1}\right)$.
\textit{This is where the $t\leq\frac{n}{4}$ assumption comes into
play.}

With some abuse of notation, we assume $d_{1},\dots,d_{i}$ are already
in $a_{i}^{-1}a_{i+1}\left\langle e_{t+1},\dots,e_{n}\right\rangle \cap\left\langle e_{t+1},\dots,e_{n}\right\rangle $.
Now we add $\left\{ w_{j}'\right\} _{j=1}^{t}\in a_{i}^{-1}a_{i+1}\left\langle e_{t+1},\dots,e_{n}\right\rangle \cap\left\langle e_{t+1},\dots,e_{n}\right\rangle $
to all the vectors such that 
\begin{align*}
\left\{ \pi_{2}\left(a_{i+1}^{-1}a_{i}\left(d_{j}+w_{j}'\right)\right)\right\} _{j=1}^{i}\cup & \left\{ \pi_{2}\left(a_{i+1}^{-1}a_{i}\left(a_{i}^{-1}a_{j}v_{j}+w_{j}+w_{j}'\right)\right)\right\} _{j=i+1}^{t}
\end{align*}
is a linearly independent set. We set $w_{j}''\leftarrow a_{i+1}^{-1}a_{i}\left(w_{j}+w_{j}'\right)$,
and $d_{j}\leftarrow d_{j}+w_{j}''$. Thus $d_{j}\in\left\langle e_{t+1},\dots,e_{n}\right\rangle $
and 
\[
a_{i+1}^{-1}a_{i}\left(a_{i}^{-1}a_{i+1}v_{i+1}+w_{j}+w_{j}'\right)=v_{i+1}+w_{j}''\in\left\langle e_{t+1},\dots,e_{n}\right\rangle 
\]
and for $j\leq i+1\leq n$, 
\[
a_{i+1}^{-1}a_{i}\left(a_{i}^{-1}a_{j}v_{j}+w_{j}+w_{j}'\right)=a_{i+1}^{-1}a_{j}v_{j}+w_{j}''.
\]
We set $d_{i+1}=v_{i+1}+w_{j}''$, and proceed inductively. As each
step takes $O\left(t\right)$ operations in $A$ and there are $t$
steps, the number of operations required to construct $b$ is $O\left(t^{2}\right)$. 
\end{proof}
\begin{lem}
\label{Lemma8}There exists $c_{t}\in A^{O\left(t^{2}\right)}$ such
that $\left\{ \pi_{2}\left(c_{t}e_{i}\right)\right\} _{i=1}^{t}$
is a linearly independent set. 
\end{lem}

\begin{proof}
We prove this lemma by induction on $i$, where $i$ is the maximum
number such that $\left\{ \pi_{2}\left(c_{i}e_{j}\right)\right\} _{j=1}^{i}$
is linearly independent. For the base case $i=1$, since $A$ is a
generating set, $\dim\left(A^{t}\left\langle e_{1}\right\rangle \right)>t$,
there exists $c_{1}\in A^{t}$ such that $\pi_{2}\left(c_{1}e_{1}\right)\neq0$.

For the induction step, assume we have constructed $c_{i}$ such that
$\left\{ \pi_{2}\left(c_{i}e_{j}\right)\right\} _{j=1}^{i}$ is linearly
independent. If $\pi_{2}\left(c_{i}e_{i+1}\right)$ is linearly independent
of $\left\{ \pi_{2}\left(c_{i}e_{j}\right)\right\} _{j=1}^{i}$, then
we are done. Otherwise, $\pi_{2}\left(c_{i}e_{i+1}\right)=\sum_{j=1}^{i}\alpha_{j}\pi_{2}\left(c_{i}e_{j}\right)$
for some $\left\{ \alpha_{j}\right\} _{j=1}^{i}\in\mathbb{F}_{p}^{n}$.
As $c_{i}e_{i+1}$ is linearly independent of $\left\{ c_{i}e_{j}\right\} _{j=1}^{i}$,
$\pi_{1}\left(c_{i}e_{i+1}\right)\neq\pi_{1}\left(c_{i}\sum_{j=1}^{i}\alpha_{j}e_{j}\right)$.

Let $z=\pi_{1}\left(c_{i}e_{i+1}\right)-\pi_{1}\left(c_{i}\sum_{j=1}^{i}\alpha_{j}e_{j}\right)$.
Since $\dim\left(A^{t}\left\langle z\right\rangle \right)\geq t$,
there exists $\tilde{c}_{i+1}\in A^{t}$ such that $\pi_{2}\left(\tilde{c}_{i+1}z\right)\neq0$.
Consider the vector space $V\subseteq\left\langle e_{t+1},\dots,e_{n}\right\rangle $
spanned by the set $\left\{ \pi_{2}\left(\tilde{c}_{i+1}e_{1}\right),\dots,\pi_{2}\left(\tilde{c}_{i+1}e_{t}\right)\right\} $
which has a basis $\left\{ v_{1},\dots,v_{\ell}\right\} $.

Let $W=\left\langle e_{t+1},\dots,e_{n}\right\rangle \cap\tilde{c}_{i+1}^{-1}\left\langle e_{t+1},\dots,e_{n}\right\rangle $,
by Lemma \ref{lemma9}, $\dim\left(W\right)\geq2n-t\geq2t$. So there
exist $\left\{ w_{1},\dots,w_{i}\right\} \subseteq W$ such that $\left\{ v_{1},\dots,v_{\ell},w_{1},\dots,w_{i}\right\} \subseteq\left\langle e_{t+1},\dots,e_{n}\right\rangle $
is a linearly independent set.

Let $M\in I_{t}\otimes\mathrm{GL}_{n-t}\left(\mathbb{F}_{p}\right)$
be a linear transformation such that $M\pi_{2}\left(c_{i}e_{j}\right)=\tilde{c}_{i+1}^{-1}w_{j}$
for every $1\leq j\leq i$. For $1\leq j\leq i$, 
\begin{align*}
\pi_{2}\left(\tilde{c}_{i+1}Mc_{i}e_{j}\right) & =\pi_{2}\left(\tilde{c}_{i+1}\left(\pi_{1}\left(Mc_{i}e_{j}\right)+\pi_{2}\left(Mc_{i}e_{j}\right)\right)\right)\\
 & =\pi_{2}\tilde{c}_{i+1}\pi_{1}Mc_{i}e_{j}+\pi_{2}\tilde{c}_{i+1}\pi_{2}Mc_{i}e_{j}
\end{align*}
Since $M\in I_{t}\otimes\mathrm{GL}_{n-t}\left(\mathbb{F}_{p}\right)$,
\[
\pi_{2}\left(\tilde{c}_{i+1}Mc_{i}e_{j}\right)=\pi_{2}\tilde{c}_{i+1}\pi_{1}c_{i}e_{j}+\pi_{2}\tilde{c}_{i+1}\pi_{2}Mc_{i}e_{j}.
\]
Since $M\pi_{2}c_{i}e_{j}=\tilde{c}_{i+1}^{-1}w_{j}$, 
\begin{align*}
\pi_{2}\left(\tilde{c}_{i+1}Mc_{i}e_{j}\right) & =\pi_{2}\tilde{c}_{i+1}\pi_{1}c_{i}e_{j}+\pi_{2}w_{j}\\
 & =\pi_{2}\tilde{c}_{i+1}\pi_{1}c_{i}e_{j}+w_{j}\in V\oplus W
\end{align*}
On the other hand, 
\[
\pi_{2}\left(\tilde{c}_{i+1}Mc_{i}e_{i+1}\right)=\pi_{2}\tilde{c}_{i+1}\pi_{1}c_{i}e_{i+1}+\sum_{j=1}^{i}\alpha_{j}w_{j}\in V\oplus W
\]
Since $w_{j}$ are linearly independent and $V\cap W=\left\{ 0\right\} $
then $\left\{ \pi_{2}\left(\tilde{c}_{i+1}Mc_{i}e_{j}\right)\right\} _{j=1}^{i}$
is a linearly independent set.

Suppose towards contradiction that $\pi_{2}\left(\tilde{c}_{i+1}Mc_{i}e_{i+1}\right)$
is linearly dependent on $\left\{ \pi_{2}\left(\tilde{c}_{i+1}Mc_{i}e_{j}\right)\right\} _{j=1}^{i}$,
then there exist $\left\{ \beta_{j}\right\} _{j=1}^{i}\subseteq\mathbb{F}_{p}$
such that 
\[
\pi_{2}\left(\tilde{c}_{i+1}Mc_{i}e_{i+1}\right)=\sum_{j=1}^{i}\beta_{j}\pi_{2}\left(\tilde{c}_{i+1}Mc_{i}e_{j}\right)
\]
In particular, since $V\oplus W$ is a direct sum of vector spaces
\[
\sum_{j=1}^{i}\beta_{j}w_{j}=\sum_{j=1}^{i}\alpha_{j}w_{j}
\]
thus $\beta_{j}=w_{j}$. On the other hand, 
\[
\pi_{2}\tilde{c}_{i+1}\pi_{1}c_{i}e_{i+1}=\sum_{j=1}^{i}\beta_{j}\pi_{2}\tilde{c}_{i+1}\pi_{1}c_{i}e_{j}
\]
thus 
\begin{align*}
0 & =\pi_{2}\left(\tilde{c}_{i+1}\pi_{1}c_{i}e_{i+1}-\sum_{j=1}^{i}\beta_{j}\tilde{c}_{i+1}\pi_{1}c_{i}e_{j}\right)\\
 & =\pi_{2}\left(\tilde{c}_{i+1}\pi_{1}c_{i}e_{i+1}-\sum_{j=1}^{i}\alpha_{j}\tilde{c}_{i+1}\pi_{1}c_{i}e_{j}\right)\\
 & =\pi_{2}\left(\tilde{c}_{i+1}z\right)\neq0
\end{align*}
So for $c_{i+1}=\tilde{c}_{i+1}Mc_{i}$, the set $\left\{ \pi_{2}\left(c_{i+1}e_{j}\right)\right\} _{j=1}^{i+1}$
is linearly independent. As every iteration takes $O(t)$ steps and
there are $t$ such iterations, it takes $O(t^{2})$ steps to construct
$c_{t}$. 
\end{proof}
Our goal is to modify $a$ so that in addition to $ae_{i}=e_{t+i}$
for $1\leq i\leq t$, we will also have $ae_{t+i}=e_{i}$. To achieve
this, we first prove the following lemma, 
\begin{lem}
\label{lema} Let $A\in\mathrm{Mat}_{t\times t}\left(\mathbb{F}_{p}\right)$
and $B\in\mathrm{Mat}_{t\times n-2t}\left(\mathbb{F}_{p}\right)$,
where $\begin{pmatrix}A & B\end{pmatrix}$ has full column rank; then
there exists $J\in\mathrm{Mat}_{n-t\times t}\left(\mathbb{F}_{p}\right)$
such that $A+BJ$ has full column rank. 
\end{lem}

\begin{proof}
We prove by induction on $t-\dim\left(\mathrm{col}\left(A\right)\right)$.
If $\dim\left(\mathrm{col}\left(A\right)\right)=t$, then we are done.
Otherwise, there exists $A_{i}$ such that $A_{i}\in\mathrm{Span}\left\{ A_{j}\right\} _{j\neq i}$;
without loss of generality, assume $i=1$. Since $\dim\left(\mathrm{col}\left(A\right)\right)<t$,
there exists $B_{k}$ for $1\leq k\leq n-2t$ such that $\dim\left(\mathrm{col}\left(A\right)+\left\langle B_{k}\right\rangle \right)=\dim\left(\mathrm{col}\left(A\right)\right)+1$.
Thus for $J=E_{k1}$, 
\[
\dim\mathrm{col}\begin{pmatrix}A_{1}+B_{k} & A_{2} & \dots & A_{t}\end{pmatrix}=\dim\mathrm{col}\left(A+BJ\right)=\dim\left(\mathrm{col}\left(A\right)\right)+1.
\]
\end{proof}
\begin{lem}[Swapping Lemma]
\label{swapping_Lemma} 
\[
A_{\mathrm{swap}}=\begin{pmatrix}0 & I_{t\times t} & 0\\
I_{t\times t} & 0 & 0\\
0 & 0 & I_{n-2t\times n-2t}
\end{pmatrix}\in A^{O\left(t^{2}\right)}
\]
\end{lem}

\begin{proof}
In Lemma \ref{main Lemma}, we constructed $a\in A^{O(t^{2})}$ which
such that $ae_{i}=e_{t+i}$ for every $1\leq i\leq t$. Thus the matrix
$a$ is of the form 
\[
a=\begin{pmatrix}0_{t\times t} & A_{t\times t} & B_{t\times n-2t}\\
I_{t\times t} & C_{t\times t} & D_{t\times n-2t}\\
0_{n-2t\times t} & E_{n-2t\times t} & F_{n-2t\times n-2t}
\end{pmatrix}
\]
where $\begin{pmatrix}A & B\end{pmatrix}$ has full column rank and
$\begin{pmatrix}E & F\end{pmatrix}$ also has full column rank. Since
$\begin{pmatrix}A & B\end{pmatrix}$ has full column rank there exists
$X_{1}\in\mathrm{Mat}_{n-t\times t}\left(\mathbb{F}_{p}\right)$ such
that $A+BX_{1}$ has rank $t$. Applying $\begin{pmatrix}I_{t} & 0 & 0\\
0 & I_{t} & 0\\
0 & X_{1} & I_{n-2t}
\end{pmatrix}$ from the right 
\begin{align*}
a & \leftarrow\begin{pmatrix}0_{t\times t} & A_{t\times t}+B_{t\times n-2t}X_{1} & B_{t\times n-2t}\\
I_{t\times t} & C_{t\times t}+D_{t\times n-2t}X_{1} & D_{t\times n-2t}\\
0_{n-2t\times t} & E_{n-2t\times t}+F_{n-2t\times n-2t}X_{1} & F_{n-2t\times n-2t}
\end{pmatrix}\\
 & =\begin{pmatrix}0_{t\times t} & A_{t\times t}' & B_{t\times n-2t}\\
I_{t\times t} & C_{t\times t}' & D_{t\times n-2t}\\
0_{n-2t\times t} & E_{n-2t\times t}' & F_{n-2t\times n-2t}
\end{pmatrix}
\end{align*}
Now $A_{t\times t}'$ is full rank. We apply $\begin{pmatrix}I_{t\times t} & 0 & 0\\
0 & \left(A_{t\times t}'\right)^{-1} & 0\\
0 & 0 & I_{n-2t\times n-2t}
\end{pmatrix}$ from the right. 
\[
a\leftarrow\begin{pmatrix}0_{t\times t} & I_{t\times t} & B_{t\times n-2t}\\
I_{t\times t} & C_{t\times t}'\left(A_{t\times t}'\right)^{-1} & D_{t\times n-2t}\\
0_{n-2t\times t} & E_{n-2t\times t}'\left(A_{t\times t}'\right)^{-1} & F_{n-2t\times n-2t}
\end{pmatrix}
\]
renaming 
\begin{align*}
C'' & =C'A'^{-1}\\
E'' & =E'A'^{-1}.
\end{align*}
We have 
\[
a\leftarrow\begin{pmatrix}0_{t\times t} & I_{t\times t} & B_{t\times n-2t}\\
I_{t\times t} & C_{t\times t}'' & D_{t\times n-2t}\\
0_{n-2t\times t} & E_{n-2t\times t}'' & F_{n-2t\times n-2t}
\end{pmatrix}
\]
Now we apply 
\[
\begin{pmatrix}I_{t\times t} & 0 & 0\\
0 & I_{t\times t} & -B_{t\times n-2t}'\\
0 & 0 & I_{n-2t\times n-2t}
\end{pmatrix}
\]
from the right. 
\begin{align*}
a & \leftarrow\begin{pmatrix}0_{t\times t} & I_{t\times t} & B_{t\times n-2t}-B_{t\times n-2t}\\
I_{t\times t} & C_{t\times t}'' & D_{t\times n-2t}-C_{t\times t}''B_{t\times n-2t}\\
0_{n-2t\times t} & E_{n-2t\times t}'' & F_{n-2t\times n-2t}-E_{t\times t}''B_{t\times n-2t}
\end{pmatrix}\\
 & =\begin{pmatrix}0_{t\times t} & I_{t\times t} & 0_{t\times t}\\
I_{t\times t} & C_{t\times t}'' & D_{t\times n-2t}-C_{t\times t}''B_{t\times n-2t}\\
0_{n-2t\times t} & E_{n-2t\times t}'' & F_{n-2t\times n-2t}-E_{t\times t}''B_{t\times n-2t}
\end{pmatrix}
\end{align*}
As $Y=\begin{pmatrix}I_{t\times t} & D_{t\times n-2t}-C_{t\times t}''B_{t\times n-2t}\\
0_{n-2t\times t} & F_{n-2t\times n-2t}-E_{t\times t}''B_{t\times n-2t}
\end{pmatrix}$ has full row rank we can apply $\begin{pmatrix}I & 0\\
0 & Y_{n-t\times n-2t}^{-1}
\end{pmatrix}$ from the left to get 
\[
a\leftarrow\begin{pmatrix}0_{t\times t} & I_{t\times t} & 0_{t\times t}\\
I_{t\times t} & X & 0_{t\times n-2t}\\
0_{n-2t\times t} & Z & I_{n-2t\times n-2t}
\end{pmatrix}
\]
from some $X$ and $Z$. Apply 
\[
\begin{pmatrix}I_{t} & 0 & 0\\
0 & I_{t} & 0\\
0 & -Z & I_{n-2t}
\end{pmatrix}
\]
from the right to obtain 
\[
a\leftarrow\begin{pmatrix}0_{t\times t} & I_{t\times t} & 0_{t\times t}\\
I_{t\times t} & X & 0_{t\times n-2t}\\
0_{n-2t\times t} & 0 & I_{n-2t\times n-2t}
\end{pmatrix}
\]
Restricting our focus on the first $3t$ indices of the matrix 
\[
\tilde{a}=\begin{pmatrix}0 & I_{t\times t} & 0\\
I_{t\times t} & X & 0\\
0 & 0 & I_{t\times t}
\end{pmatrix}
\]
as we could construct the matrix $\tilde{a}$, we can also construct
$\tilde{a}^{-1}$ by applying the inverse transformations in reverse
order 
\[
\tilde{a}^{-1}=\begin{pmatrix}-X & I_{t\times t} & 0\\
I_{t\times t} & 0 & 0\\
0 & 0 & I_{t\times t}
\end{pmatrix}.
\]
Conjugating the constructable matrix (as it is in the groumvirate)
\[
b=\begin{pmatrix}I_{t\times t} & 0 & 0\\
0 & I_{t\times t} & -X\\
0 & 0 & I_{t\times t}
\end{pmatrix}
\]
by the matrix $\tilde{a}^{-1}$, 
\[
\tilde{a}^{-1}b\tilde{a}=\begin{pmatrix}I_{t\times t} & 0 & -X\\
0 & I_{t\times t} & 0\\
0 & 0 & I_{t\times t}
\end{pmatrix}.
\]
Applying a permutation matrix $P$ from both sides which swaps the
$\left\{ t+1,\dots,2t\right\} $ columns with the $\left\{ 2t+1,\dots,3t\right\} $,
we get a matrix 
\[
P\tilde{A}^{-1}B\tilde{A}P=\begin{pmatrix}I_{t\times t} & -X & 0\\
0 & I_{t\times t} & 0\\
0 & 0 & I_{t\times t}
\end{pmatrix}
\]
and 
\[
\tilde{A}P\tilde{A}^{-1}B\tilde{A}P=\begin{pmatrix}0 & I_{t\times t} & 0\\
I_{t\times t} & X & 0\\
0 & 0 & I_{t\times t}
\end{pmatrix}\begin{pmatrix}I_{t\times t} & -X & 0\\
0 & I_{t\times t} & 0\\
0 & 0 & I_{t\times t}
\end{pmatrix}=\begin{pmatrix}0 & I_{t\times t} & 0\\
I_{t\times t} & 0 & 0\\
0 & 0 & I_{t\times t}
\end{pmatrix}
\]
which is the desired swap matrix. 
\end{proof}
\begin{proof}[Proof of Theorem \ref{useful Theorem}]
By Lemma \ref{swapping_Lemma}, the matrix $A_{\mathrm{swap}}$ can
be constructed in $O\left(t^{2}\right)$ steps. Combining the swap
transformation and any permutation on the last $n-t$ vectors, we
obtain a transitive action on $V'$. The distance between any two
vertices in $V'$ is $O(t^{2})$ since each step in the proof requires
at most $O(t^{2})$ multiplications. 
\end{proof}

\section{Constructing the Matrix Group via Its Bruhat Decomposition}

Our proof relies on the following decomposition of $\mathrm{GL}_{n}\left(\mathbb{F}_{p}\right)$
known in the literature as the Bruhat decomposition (a reference for
the case $\mathrm{SL}_{n}\left(\mathbb{F}_{p}\right)$ is proven in
\cite{bruhat}). 
\begin{thm}
$\mathrm{GL}_{n}\left(\mathbb{F}_{p}\right)=BWB$ where $B\subseteq\mathrm{GL}_{n}\left(\mathbb{F}_{p}\right)$
is the group of lower-triangular matrices and $W\subseteq\mathrm{GL}_{n}\left(\mathbb{F}_{p}\right)$
is the set of monomial matrices (the set of matrices in which each
row and column contains exactly one non-zero element). 
\end{thm}

To construct a matrix $M\in\mathrm{GL}_{n}\left(\mathbb{F}_{p}\right)$,
we construct its Bruhat decomposition $b_{1},b_{2}\in B$ and $w\in W$
and we compose all elements to get $M=b_{1}wb_{2}$. 
\begin{lem}
The triangular matrices can be constructed in $O\left(n^{2}\right)$
steps, i.e., $B\subseteq A^{O\left(n^{2}\right)}$. 
\end{lem}

\begin{proof}
Let $L\in B$. Using Theorem \ref{useful Theorem}, we construct a
primitive which modifies the $\left\{ 1,\dots,\frac{n}{3}\right\} $
columns for row indices $J\subseteq\left\{ \frac{n}{3}+1,\dots,n\right\} $
where $\left|J\right|=\frac{n}{3}$. Applying this primitive 2 times
enables us to construct the $\left\{ 1,\dots,\frac{n}{3}\right\} $
columns of the matrix. Using Lemma \ref{corr}, we can construct the
rest of the columns using a matrix in $A^{4}$. 
\end{proof}
\begin{lem}
The monomial matrices can be constructed in $O\left(n^{2}\right)$
steps, i.e., $W\subseteq A^{O\left(n^{2}\right)}$. 
\end{lem}

\begin{proof}
Using Theorem \ref{useful Theorem}, we construct a matrix which swaps
the $\lbrace1,\dots,t\rbrace$ columns with the $\lbrace\sigma(1),\dots,\sigma(t)\rbrace$
columns in $A^{O(t^{2})}$ steps. Now all of the first $t$ columns
are fixed in position. All that remains is to set the remaining $n-t$
columns in place. By Corollary~\ref{finalcor}, any permutation of
the last $n-t$ columns can be constructed in $O(t^{2})$ steps. 
\end{proof}

\section{A Matching Lower Bound on the Diameter}

Let $t=Cn$ for $0<C<1$. We define the generating set 
\[
A=\left\{ \begin{pmatrix}I_{t\times t} & 0\\
0 & X
\end{pmatrix}:X\in\mathrm{GL}_{n-t}\left(\mathbb{F}_{p}\right)\right\} \bigcup\left\{ e_{i}\leftrightarrow e_{i+1}:i=1,\dots,t\right\} ,
\]
where $e_{i}\leftrightarrow e_{i+1}$ is the swap between $e_{i}$
and $e_{i+1}$ which fixes the other basis vectors. This is a generating
set because it generates the swap matrix (as transpositions generate
the permutations). Moreover, 
\[
\left|A\right|=\Theta\left(p^{\left(n-t\right)^{2}}\right)=\Theta\left(\left(p^{n^{2}}\right)^{\left(\frac{n-t}{n}\right)^{2}}\right)\leq|G|^{1-c}
\]
for an appropriately chosen $c$. Let $B=\begin{pmatrix}0 & I_{t\times t} & 0\\
I_{t\times t} & 0 & 0\\
0 & 0 & I_{n-2t\times n-2t}
\end{pmatrix}$. Assume $B=a_{k}a_{k-1}\dots a_{1}$, where $a_{i}\in A$. Define
$A_{0}=e$, and $A_{l}=a_{l}a_{l-1}\dots a_{1}$ for $1\leq l\leq k$.
We set $\mathcal{F}_{0}=\left\{ 1,\dots,t\right\} $ and $\mathcal{F}_{l}=\left\{ i\in\left\{ 1,\dots,t\right\} :A_{j}e_{i}\in\left\{ e_{1},\dots,e_{t}\right\} \forall1\leq j\leq l\right\} $.
Let 
\[
d_{i,l}=\begin{cases}
t+1-j & \text{if }A_{l}e_{i}=e_{j}\text{ for }j\in\left\{ 1,\dots,t+1\right\} \wedge i\in\mathcal{F}_{l}\\
0 & \text{otherwise}
\end{cases}
\]
and 
\[
d_{l}=\sum_{i=1}^{t}d_{i,l}.
\]
We prove that $d_{l+1}\geq d_{l}-1$ for every $1\leq l\leq k-1$.
Assume $a_{l+1}$ is a swap between $e_{i}$ and $e_{i+1}$ then $d_{l}\neq d_{l+1}$
iff $a_{l+1}$ swaps $t$ and $t+1$, thus $d_{l+1}=d_{l}-1$. Otherwise
$a_{l+1}$ acts on $\left\{ e_{t+1},\dots,e_{n}\right\} $ so it cannot
change $d_{i,l}$ for any $i$. Therefore 
\[
d_{k}\geq d_{0}-k=\sum_{i=1}^{t}i-k=\begin{pmatrix}t+1\\
2
\end{pmatrix}-k.
\]
But $d_{k}=0$ since $B$ is the matrix that swaps $e_{1},\dots,e_{t}$
with $e_{t+1},\dots,e_{2t}$, and thus 
\[
k\geq\begin{pmatrix}t+1\\
2
\end{pmatrix}=\Omega\left(n^{2}\right).
\]

\subsection*{Acknowledgments}

We would like to thank Noam Lifshitz for his guidance, Elad Tzalik
for helpful discussions, and Edan Orzech for reviewing the manuscript.
\bibliographystyle{plain}

\end{document}